\documentclass[]{amsart}
\usepackage{mathrsfs}
\usepackage{graphicx}
\usepackage[T1]{fontenc}
\usepackage[utf8]{inputenc}
\usepackage{amssymb}
\usepackage{amsmath}
\usepackage{amsthm}
\usepackage{galois}
\usepackage[all,cmtip]{xy}

\allowdisplaybreaks                 
\usepackage{float}                 
\restylefloat{figure}              
\usepackage{caption2}              
\usepackage{array}
\newtheorem{thm}{Theorem}[section]

\newtheorem{lem}[thm]{Lemma}
\newtheorem{prop}[thm]{Proposition}
\newtheorem{ques}[thm]{Question}

\theoremstyle{definition}
\newtheorem{defn}[thm]{Definition}
\newtheorem{exam}[thm]{Example}
\theoremstyle{remark}
\newtheorem{rem}[thm]{Remark}
\numberwithin{equation}{section}

\newcommand{\Z}{\mathbb Z}
\newcommand{\Q}{\mathbb Q}
\newcommand{\R}{\mathbb R}

\newcommand{\fix}{\mathrm{Fix}}

\newcommand{\rk}{\mathrm{rk}}
\newcommand{\id}{\mathrm{Id}}

\newcommand{\iso}{\mathrm{Isom}}
\newcommand{\aut}{\mathrm{Aut}}

\newcommand{\fgfp}{\mathrm{FGFP}}

\newcommand{\fgfpa}{\mathrm{FGFP_a}}

\begin{document}

\title{A note on the finitely generated fixed subgroup property}
\author{Jialin Lei}
\address{School of Mathematics and Statistics, Xi'an Jiaotong University, Xi'an 710049, CHINA}
\email{leijialin0218@stu.xjtu.edu.cn\\https://orcid.org/0000-0002-2470-8535}

\author{Jiming Ma}
\address{School of Mathematical Sciences, Fudan University, Shanghai 200433, CHINA}
\email{majiming@fudan.edu.cn\\https://orcid.org/0009-0001-3065-4094}

\author{Qiang Zhang}
\address{School of Mathematics and Statistics, Xi'an Jiaotong University, Xi'an 710049, CHINA}
\email{zhangq.math@mail.xjtu.edu.cn\\https://orcid.org/0000-0001-6332-5476}

\date{\today}


\thanks{The authors are partially supported by NSFC (Nos. 12171092 and 12271385), and the Shaanxi Fundamental Science Research Project for Mathematics and Physics (Grant No. 23JSY027).}


\begin{abstract}
We study when a group of form $G\times\mathbb{Z}^m (m\geq 1)$ has the finitely generated fixed subgroup property of automorphisms ($\fgfp_a$), by using the BNS-invariant, and provide some partial answers and non-trivial examples.
\end{abstract}
\maketitle

\section{Introduction}

For a group $G$, the \emph{rank} of $G$, denoted $\rk(G)$, is the minimal number of generators of $G$, and $\aut(G)$ denotes the group of all automorphisms of $G$.  For an endomorphism $\phi$ of $G$, the \emph{fixed subgroup} of $\phi$ is defined to be
$$\fix\phi :=\{g\in G \mid \phi(g)=g\}.$$
The study of fixed subgroups goes back to Dyer and Scott in 1975. In the paper  \cite{DS75},  they proved that for a finite order automorphism $\phi$ of a free group $F_n$ of rank $n$, the rank $\rk(\fix\phi)$ is not greater than $n$. Moreover, Scott conjectured that $\rk(\fix\phi)\leq n$ for any $\phi\in \aut(F_n)$, which was known as the Scott conjecture solved by Bestvina and Handel \cite{BH92}  in 1988, and extended to all endomorphisms by Imrich and Turner \cite{IT89} almost simultaneously. Note that for every endomorphism $\phi$ of a \textit{surface group} $G$ (i.e., the fundamental group of a closed surface), the same bound $\rk(\fix\phi)\leq \rk(G)$ also holds, see \cite{JWZ11}.  The study on fixed subgroups of various groups and related topics, such as Nielsen fixed point theory, has been active to nowadays, producing many interesting results, e.g. \cite{CL22, JZ24, RV21, ZZ21, ZZ23, ZZ24},  etc.

More generally, we say that a group $G$ has the \emph{finitely generated fixed subgroup property} of automorphisms ($\fgfp_a$), if  the fixed subgroup $\fix\phi$ is finitely generated for every automorphism $\phi\in \aut(G)$. Note that if a group $G$ has  $\fgfpa$, it must be itself finitely generated. In addition to free groups and surface groups mentioned above, many types of groups had been proven to have $\fgfpa$, such as Gromov hyperbolic groups \cite{Pa89, Neu92, Sh99},  limit groups \cite{MO12}, etc. The $\fgfpa$ property is preserved under taking free products \cite{LZ23}, but not under taking direct products \cite{ZVW15}. For example, the free group $F_2$ and $\Z$ both have $\fgfpa$, but their direct product $F_2\times \Z$  does not:
\begin{exam}
Let $\phi$ be an automorphism of $F_2\times \Z=\langle a, b\rangle \times \langle t \rangle$ defined by $\phi(a)=at, \phi(b)=b$ and $\phi(t)=t$.  An element $u\in\fix\phi$ if and only if the total $a$ exponent in $u$ is zero. Then the fixed subgroup $\fix\phi\cong F_{\aleph_{0}}\times \Z$ is not finitely generated (generated by the set $\{ t, a^iba^{-i}\mid i\in \Z\}$).
\end{exam}
So we wonder:
\begin{ques}\label{problem 1}
For a group $G$, when does $G\times \mathbb{Z}^m (m\geq 1)$ have $\fgfpa$?
\end{ques}
In this note, we will investigate this question and provide some partial answers, see Theorem \ref{main thm 1}, Theorem \ref{main thm 2} and Theorem \ref{main thm 3}, by using the BNS-invariant in Section \ref{sect 2} and Section \ref{sect 3}.

\section{Preliminaries}\label{sect 2}

\subsection{BNS-invariant}
The BNS-invariant introduced by Bieri, Neumann and Strebel \cite{BNS87} in 1987 is a geometric invariant of finitely generated groups, which is inspired by the work of Thurston \cite{Thu86}. It determines whether the kernel of a homomorphism from a group to an abelian group is finitely generated or not. Generally, the BNS-invariant is hard to compute. It was described for some families of groups like RAAGs \cite{MV95}, limit groups \cite{Koc10} and some other groups. Bieri and Renz \cite{BR88} introduced the higher dimension BNS-invariant to get more information of the kernel.

\begin{defn} Let $G$ be a finitely generated group with a finite generating set $X\subset G$,
 $n=\rk(H^1(G; \Z))$ the torsion-free rank of the abelianization of $G$, and the character sphere $S(G)=(\mathrm{Hom}(G,\mathbb{R})-{0})/\mathbb{R}_{+}$ which is an $(n-1)$-sphere. Note that an element of $S(G)$ is an equivalence class $[\chi]=\{r\chi|r\in \mathbb{R}_+\}$. Denote $\Gamma=\Gamma(G,X)$ the Cayley graph of $G$ with respect to $X$.
The first $\Sigma$-invariant (or \textit{$BNS$-invariant}) of $G$ is
$$\Sigma^1(G):=\{[\chi]\in S(G)\ |\ \Gamma_{\chi} ~\rm{is ~connected}\},$$
where $\Gamma_{\chi}$ is the subgraph of $\Gamma$ whose vertices are the elements $g\in G$ with $\chi(g)\geq 0$ and whose edges are those of $\Gamma$ which connect two such vertices.

A non-trivial homomorphism $\chi: G \to \R$ with discrete (and hence infinite cyclic) image is said to be a \textit{discrete }or \textit{rank one }homomorphism. It represents a rational point of  $S(G)$. The set
$$S\Q(G):=\{[\chi] \in S(G) \mid \chi ~\rm{is ~discrete}\}$$
of rational points is dense in $S(G$).
\end{defn}
For later use, we present the main results of the paper \cite{BNS87} in the following.
\begin{thm}[Bieri-Neumann-Strebel]\label{ker f.g}
	Let $G$ be a finitely generated group. Then
 	\begin{enumerate}
 \item Let $N$ be a normal subgroup of $G$ with $G/N$ abelian. Then $N$ is finitely generated if and only if
$$S(G, N):=\{[\chi]\in S(G) \mid \chi(N)=0\}\subset \Sigma^1(G).$$
In particular, $\Sigma^1(G)=S(G)$ if and only if the derived subgroup $[G,G]$ is finitely generated;
\item  Let $\phi: G \rightarrow \mathbb{Z}$ be a non-trivial homomorphism. Then $\ker\phi$ is finitely generated if and only if $\{\phi, -\phi\} \subset \Sigma^1(G)$. In particular, $S\Q(G)\subset \Sigma^1(G)$ if and only if $\ker\phi$ is finitely generated for every homomorphism $\phi: G \rightarrow \mathbb{Z}$.
\end{enumerate}
\end{thm}

\subsection{Automorphism}
For a centerless group $G$ (i.e. the center $C(G)$ is  trivial), we have:

\begin{prop}\label{form of aut}
If $G$ is a centerless group, then every automorphism $\phi: G\times \Z^m\rightarrow G\times \Z^m$ $(m\geq 1)$ has the following form:
\[ \phi(g,v)=(\psi(g), ~~\alpha(g)+\mathcal Lv), ~~(g, v)\in G\times \Z^m, \]
where $\psi :G\rightarrow G$ and $\mathcal L: \Z^m\to \Z^m$ are automorphisms, and $\alpha :G\rightarrow \Z^m$ is a homomorphism.
\end{prop}
\begin{proof}
Since $G$ is a centerless group, the center
$$C(G\times \Z^m)=C(G)\times C(\Z^m)=1\times \Z^m.$$ Note that an automorphism preserves the center,  so $\phi(1\times \Z^m)=1\times \Z^m$ and $\phi(1,v)=(1, \mathcal{L}v)$ for $\mathcal{L}$ an invertable matrix. Therefore, we can suppose
\[ \phi(g,v)=(\psi(g), ~~\alpha(g)+\mathcal Lv), ~~(g, v)\in G\times \Z^m. \]
The endomorphism  $\psi :G\rightarrow G$ is clearly surjective, so it remains to show that it is injective as well. Indeed, for any $g\in \ker\psi$, we have $$\phi(g,0)=(1, \alpha(g))\in 1\times \Z^m=C(G\times \Z^m).$$So $(g,0)\in 1 \times \Z^m$, it implies $g=1$ and hence   $\psi :G\rightarrow G$ is an automorphism.
\end{proof}

\section{Main results}\label{sect 3}

In this section, we study the necessary and sufficient conditions  for Question \ref{problem 1} to have a positive answer.

\subsection{Necessary condition}
\begin{thm}\label{main thm 1}
For a group $G$,
\begin{enumerate}
    \item if $G\times \mathbb{Z}$ has $\fgfpa$, then $G$ has $\fgfpa$ and  $S\Q(G)\subset \Sigma^1(G)$, or equivalently, every homomorphism $\alpha: G\to\Z$ has finitely generated kernel;
    \item  if $G\times \mathbb{Z}^m$ has $\fgfpa$ for some $m\geq \rk(H^1(G;\Z))$, then $G\times \mathbb{Z}^n$ has $\fgfpa$ for every $0\leq n\leq m$, and $\Sigma^1(G)=S(G)$, or equivalently, its derived subgroup $[G,G]$ is finitely generated.
\end{enumerate}

\end{thm}

\begin{proof}
(1) First, we assume that $G$ does not have $\fgfpa$, that is, there is an automorphism $\psi$ of $G$, such that $\fix\psi$ is not finitely generated. Consider the automorphism $\phi:G\times \Z \to G\times \Z$  given by:
$$\phi(g, n)=(\psi(g), n),$$
whose fixed subgroup $\fix\phi=\fix\psi\times \Z$ is not finitely generated, contradicting that $G\times \Z$ has $\fgfpa$.

Now we assume  $S\Q(G)\not\subset \Sigma^1(G)$,  or equivalently,  there is a non-trivial homomorphism $\alpha: G\to \Z$ such that $\ker\alpha$ is not finitely generated by Theorem \ref{ker f.g}. Let $\phi:G\times \Z \to G\times \Z$  be as follows:
$$\phi(g,n)=(g, ~\alpha(g)+n).$$
Then $\phi$ is an automorphism whose inverse is $\phi^{-1}(g,n)=(g, n-\alpha(g))$. It is easy to see that $\fix\phi=\ker\alpha\times \mathbb{Z}$ is not finitely generated, also contradicting that $G\times \Z$ has $\fgfpa$.\\

(2) If $G\times \mathbb{Z}^m$ has $\fgfpa$ for some $m\geq \rk(H^1(G;\Z))$, then $G\times \mathbb{Z}^n$ has $\fgfpa$ for every $0\leq n\leq m$, following directly from Item (1) because $G\times \Z^m=(G\times \Z^{m-1})\times \Z$.  Now we assume that $[G,G]$ is not finitely generated. Then Theorem \ref{ker f.g} implies $\Sigma^1(G)\neq S(G)$, and there is a non-trivial homomorphism $\alpha: G\to \R$ such that $\ker\alpha$ is not finitely generated.  Note that the image of $\alpha$ is an abelian group $\Z^n$ with $n\leq \rk(H^1(G;\Z))$. So $\alpha$ can be viewed as a homomorphism $\alpha: G\to \Z^n$, and the automorphism
$$\phi:G\times \Z^n \to G\times \Z^n, ~\phi(g,v)=(g, ~\alpha(g)+v),$$
has fixed subgroup $\fix\phi=\ker\alpha\times \mathbb{Z}^n$ not finitely generated, contradicting that $G\times \Z^n$ has $\fgfpa$.
\end{proof}

\begin{rem}\label{fanli}
Spahn and Zaremsky \cite{SZ21} showed that every kernel of a map from the group $F_{2,3}$ to $\Z$ is finitely generated, but there exist maps from $F_{2,3}$ to $\Z^2$ whose kernels are not finitely generated. For the definition of  $F_{2,3}$ and more details, see \cite{SZ21}.
\end{rem}

\begin{exam}
	Let $G$ be a non-abelian limit group. Then $G\times \mathbb{Z}$ (and hence $G\times \Z^m (m\geq 1)$) does not have $\fgfpa$.  Indeed, Kochloukova \cite{Koc10} proved that the BNS-invariant of a non-abelian limit group is the empty set. Note that the sphere $S(G)$ is not empty, so $S\Q(G)\not\subset \Sigma^1(G)$. Then $G\times \mathbb{Z}$ does not have $\fgfpa$ according to Theorem \ref{main thm 1}.
\end{exam}
\subsection{Sufficient condition}
To give sufficient conditions for Question \ref{problem 1} to have a positive answer,
we need introduce the Howson and weakly Howson properties first.

\begin{defn}
A group $G$ is said to have the \textit{Howson property}, if the intersection $H\cap K$ of any two finitely generated subgroups $H, K<G$ is again finitely generated; and $G$ is said to have the \textit{weakly Howson property}, if in addition, one of $H$ and $K$ is normal in $G$.
\end{defn}

Note that a group with the Howson property necessarily has the weakly Howson property; and simple groups have the weakly Howson property clearly.
Free groups and surface groups both have the Howson property \cite{How54}. More concretely, for a free or surface group $G$,
$$\rk(H\cap K)-1\leq (\rk(H)-1)(\rk(K)-1),$$
which was conjectured by Hanna Neumann in 1957, and proved independently by Friedman \cite{Fri14} and by Mineyev \cite{Min12} in 2011 for free groups, and by Antol\'{i}n and Jaikin-Zapirain \cite{AJZ22} in 2022 for surface groups.

Moreover, we have:

\begin{lem}(Some basic properties of the weakly Howson property)
\begin{enumerate}
  \item $F_2\times \Z$ does not have the weakly Howson property, and hence does not have the Howson property;
  \item The Howson property is heritable (i.e., if a group has the Howson property, then each subgroup of it does), and hence any group containing a subgroup isomorphic to $F_2\times \Z$ does not have the Howson property;
  \item Thompson's group $V$ has the weakly Howson property but does not have the Howson property;

  \item The weakly Howson property is not heritable.
  \end{enumerate}
\end{lem}

\begin{proof}
(1) Let $F_2\times \Z=\langle a, b\rangle \times \langle t \rangle$ and let $K=\langle a, bt\rangle$ be a finitely generated subgroup. Then $F_2=\langle a, b\rangle$ is normal in $F_2\times \Z$, and $F_2\cap K=\langle b^nab^{-n}\mid n\in\Z\rangle$  is the normal closure of $a$ in $F_2$, and hence not finitely generated. Therefore, $F_2\times \Z$ does not have the weakly Howson property (and hence does not have the Howson property).

(2) The Howson property is heritable clearly. So any group with a subgroup isomorphic to $F_2\times \Z$ does not have the Howson property. For example,  the special linear group $SL(n, \Z) ~( n\geq 4)$ contains a subgroup isomorphic to  $F_2\times \Z$ and does not have the Howson property. Moreover, by the virtually fibered theorem of 3-manifolds, every hyperbolic 3-manifold of finite volume is finitely covered by a surface bundle over the circle. So the fundamental group of every hyperbolic 3-manifold of finite volume does not have the Howson property either \cite{So92}.

(3) Note that Thompson's group $V$ is a finitely presented infinite simple group, then it has the weakly Howson property. Moreover, Thompson's group $V$ contains a remarkable variety of subgroups, such as finitely
generated free groups, finitely generated abelian groups, Houghton’s groups; and the class of
subgroups of $V$ is closed under direct products and restricted wreath products
with finite or infinite cyclic top group \cite{Hig74}. In particular, $V$ contains $F_2\times \Z$ as a subgroup and hence $V$ does not have the Howson property.

(4) This item follows the above item (3) clearly.
\end{proof}

\begin{lem}\label{f.g subgrp}
Let $G$ be a finitely generated group, $H<G$ a finite index subgroup and $K<G$ a finitely generated subgroup. Then $H\cap K$ is finitely generated.
\end{lem}

\begin{proof}
It is easy to see that $H\cap K$ is finite index in $K$. Since $K$ is finitely generated, we have $H\cap K$ is also finitely generated.
\end{proof}

\begin{thm}\label{main thm 2}
Let $G$ be a centerless group with the weakly Howson property. Then
\begin{enumerate}
\item $G\times \mathbb{Z}$ has $\fgfpa$ if and only if $G$ has $\fgfpa$ and  $S\Q(G)\subset \Sigma^1(G)$, or equivalently, every homomorphism $\alpha: G\to\Z$ has finitely generated kernel.
\item $G\times \mathbb{Z}^m$ has $\fgfpa$ for every $m\geq 1$ if and only if $G$ has $\fgfpa$ and  $\Sigma^1(G)=S(G)$, or equivalently, the derived subgroup $[G,G]$ is finitely generated.
\end{enumerate}
\end{thm}

\begin{proof}
The ``only if" part follows  from Theorem \ref{main thm 1} clearly.  For the  ``if" part, note that $G$ is centerless, then by  Proposition \ref{form of aut}, every automorphism $\phi$ of $G\times \mathbb{Z}^m$ $(m\geq 1)$ has the form
$$\phi(g,v)=(\psi(g),\alpha(g)+\mathcal{L}v),$$
where $\psi\in \aut(G)$,  $\alpha\in \mathrm{Hom}(G,\mathbb{Z}^m)$ and $\mathcal{L}\in \aut(\Z^m).$  So the fixed subgroup
\begin{equation}\label{eq. fixsubgp}
\fix\phi=\{(g, v)\in G\times \Z^m\mid \psi(g)=g, ~\alpha(g)+\mathcal{L}v=v\}.
\end{equation}
Below we prove the ``if" part for the two items respectively.

(1) In this case, $m=1$ and $\mathcal{L}=\pm \id$.  Since $G$ has $\fgfpa$ and $S\Q(G)\subset \Sigma^1(G)$, $\fix\psi$ and $\ker\alpha$ are both finitely generated for every homomorphism $\alpha: G\to\Z$ by Theorem \ref{ker f.g}.
When $\mathcal{L}=\id$, we have
$$\fix\phi=\{(g, n)\mid \psi(g)=g, ~\alpha(g)+n=n\}=(\fix\psi\cap\ker\alpha)\times\mathbb{Z}.$$
Then $\fix\phi$ is finitely generated by the weakly Howson property of $G$.
When $\mathcal{L}=-\id$, the fixed subgroup is
\begin{eqnarray}
    \fix\phi&=&\{(g, n)\mid \psi(g)=g, ~\alpha(g)-n=n\}\nonumber\\
    &=&\{(g, n)\mid g\in \fix\psi\cap\alpha^{-1}(2\mathbb{Z}), ~n=\alpha(g)/2\}\nonumber\\
    &\cong&\fix\psi\cap \alpha^{-1}(2\mathbb{Z}).\nonumber
\end{eqnarray}
Actually the weakly Howson property is useless in this case, because $\alpha^{-1}(2\mathbb{Z})<G$ is a subgroup of index $\leq 2$, then  $\fix\phi$ is finitely generated by Lemma \ref{f.g subgrp}. So $G\times \Z$ has $\fgfpa$ and Item (1) holds.\\

(2) To prove $ G\times \Z^m$  has $\fgfpa$ for every $m\geq 1$, let us consider the projection $$p:  G\times \Z^m\to  G, \quad p(g, v)=g.$$
Then, by Eq. (\ref{eq. fixsubgp}), we have the natural short exact sequence
$$0\to \fix\phi\cap\ker p\hookrightarrow\fix\phi \xrightarrow{p} p(\fix\phi)\to 1,$$
where
$$\fix\phi\cap\ker p=\{(1, v)\in G\times \Z^m\mid \mathcal{L}v=v\}\cong \Z^s,$$
for some $s\leq m$, and
\begin{eqnarray}\label{eq. 1}
p(\fix\phi)&=&\{g\in \fix\psi \mid \exists v=\alpha(g)+\mathcal{L}v\in \Z^m \}\nonumber\\
&=&\fix\psi\cap\alpha^{-1}((\id-\mathcal{L})\mathbb{Z}^m)
\end{eqnarray}
is a normal subgroup of  $\fix\psi$.

By the above exact sequence, to prove that $\fix\phi$ is finitely generated, it suffices to prove that $p(\fix\phi)$ is finitely generated. Indeed, note that  $\Sigma^1(G)=S(G)$, and $G/\alpha^{-1}((\id-\mathcal{L})\mathbb{Z}^m)$ is a quotient of $\alpha(G)$ and hence it is abelian.  So by Theorem \ref{ker f.g}, $\alpha^{-1}((\id-\mathcal{L})\mathbb{Z}^m)$ is a finitely generated normal subgroup of $G$. Moreover,  since $G$ has $\fgfpa$, both of $G$ and $\fix\psi$ are finitely generated.  By the weakly Howson property of $G$, $p(\fix\phi)$ is again finitely generated.

Therefore, $\fix\phi$ is finitely generated, and hence $G\times \mathbb{Z}^m$ $(m\geq 1)$ has $\fgfpa$.
\end{proof}

For arbitrary group $G$, it seems too difficult to guarantee that $G$ has both the (weakly) Howson and $\fgfpa$ properties, and $[G, G]$ finitely generated, except $G$ is \textit{slender}, that is, every subgroup of $G$ is finitely generated. For example, all finite  groups and all finitely generated nilpotent groups are slender. Note that nilpotent groups have non-trivial centers.

Hyperbolic groups have $\fgfpa$ and non-elementary hyperbolic groups are centerless. However, a hyperbolic group $G$ is not necessarily with the Howson property and $[G,G]$ finitely generated. For example, every non-solvable surface group is hyperbolic and has the Howson property but its derived subgroup is not finitely generated.

Every finitely generated non-abelian simple group $G$ (e.g.  Thompson's group $V$) is centerless, with $[G, G]$ finitely generated and satisfies the weakly Howson property, but currently, little is known about its fixed subgroups. So we have:

\begin{ques}
Is there a centerless, infinite group $G$ with $[G, G]$ finitely generated and satisfying the (weakly) Howson and $\fgfpa$ properties?
\end{ques}

\subsection{Non-trivial example}

For a hyperbolic 3-manifold $M$,  Soma \cite{So92} showed that the fundamental group $\pi_1(M)$ has the Howson property if and only if $M$ has infinite volume.
Although the fundamental group $G=\pi_1(M)$ of a hyperbolic 3-manifold $M$ with infinite volume has the Howson property, it never has finitely generated  derived subgroup $[G, G]$ except $G$ is nilpotent. In fact, assume there is a hyperbolic 3-manifold $M$ with infinite volume and with fundamental group having finitely generated derived subgroup $[G, G]$. Moreover, assume that  $M$ is not the  solid torus, and not homotopic to a closed surface. For such  a hyperbolic 3-manifold $M$, we know one component of $\partial M$ has genus at least two, so the first homology group of $M$ is infinite.  Let $N$ be the infinite  regular cover of $M$  with fundamental group  $\pi_1(N)=[G, G]$.    We may deform the hyperbolic structure on $M$ such that it is geometrically finite and has no cusps. By Thurston's theorem \cite{Ca94},  the hyperbolic structure on $N$ is also geometrically finite.  Since  $[G, G]$ is a normal subgroup of $G$, they have the same limit set. Let $C$ be the convex hull of the limit set, then the convex core of $N$ is $C/[G, G]$, and the convex core of $M$ is $C/G$.
Both of them are non-zero finite volume since $N$ and $M$ are geometrically finite. But
 $C/[G, G]$ is an infinite cover of $C/G$, a contradiction.

In the case of $M$ having finite volume, Lin and Wang \cite{LW14} studied the fixed subgroups of automorphisms of  $\pi_1(M)$ and showed the following.

\begin{thm}\cite[Theorem 1.6]{LW14}\label{LW thm}
Suppose $G=\pi_1(M)$ where $M$ is an orientable hyperbolic 3-manifold with finite volume, and $\phi$ is an automorphism of $G$. Then $\fix\phi$ is of one of the following types: the whole group $G$; the trivial group 1; $\mathbb{Z}$; $\mathbb{Z}\times\mathbb{Z}$; a surface group $\pi_1(S)$, where $S$ can be orientable or not, and closed or not. More precisely,
	\begin{enumerate}
		\item Suppose $\phi$ is induced by an orientation preserving isometry.
		\begin{enumerate}
			\item $\fix\phi$ is either $\mathbb{Z}$, or $\mathbb{Z}\times\mathbb{Z}$, or $G$, or 1; moreover
			\item If $M$ is closed, then $\fix\phi$ is either $\mathbb{Z}$ or $G$;
		\end{enumerate}
		\item Suppose $\phi$ is induced by an orientation reversing isometry $f$.
		\begin{enumerate}
			\item If $\phi^2\neq id$, then $\fix\phi$ is either $\mathbb{Z}$ or 1;
			\item If $\phi^2= id$, then $\fix\phi$ is either 1, or the surface group $\pi_1(S)$, where $S$ is an embedded surface in $M$ that is pointwise fixed by $f$.
		\end{enumerate}
	\end{enumerate}
\end{thm}

Note that in \cite{LW14}, a 3-manifold $M$ is \textit{hyperbolic} if $M$ is orientable, compact and the interior of $M$ admits a complete hyperbolic structure of finite volume (then $M$ is either closed or with boundary consisting of a union of tori). As a corollary, we have:

\begin{prop}
\label{LW result}
 Let $G=\pi_1(M)$ for $M$ an orientable hyperbolic 3-manifold with finite volume and without an involution. Then for any automorphism $\phi$ of $G$, the fixed subgroup $\fix\phi$ is inert in $G$, namely,
$$\rk(H\cap \fix\phi)\leq \rk(H)$$
for any finitely generated subgroup $H<G$.
\end{prop}

\begin{proof}
Since $M$ has no involution, we have $\phi^2\neq id$ for any automorphism $\phi$ of $G=\pi_1(M)$,  by the well known Mostow's rigidity theorem. Then by Lin and Wang's Theorem \ref{LW thm}, the fixed subgroup $\fix\phi$ is either 1,  $\Z$,  $\mathbb{Z}\times\mathbb{Z}$ or the whole group $G$. So the conclusion holds clearly.
\end{proof}

To get non-trivial examples, we remove the weakly Howson property from the conditions in Theorem \ref{main thm 2}.

\begin{thm}\label{main thm 3}
	The group $G\times \mathbb{Z}^m$ has $\fgfpa$ for every $m\geq 1$ if $G$ is one of the following:
	\begin{enumerate}
		\item a slender group (for example, a finite group or a finitely generated nilpotent group);
		\item $G=\pi_1(M)$ where $M$ is a closed orientable hyperbolic 3-manifold with finite first homology group $H_1(M)$, and with the isometry group $\iso(M)$ of odd order.
	\end{enumerate}
\end{thm}
\begin{proof}
(1) Let $G$ be a slender group, that is, every subgroup of $G$ is finitely generated. Then $G\times \Z^m$ is again slender, and hence it has $\fgfpa$. Indeed, for any subgroup $H<G\times \Z^m$, we have a short exact sequence
$$1\to H\cap\ker(p)\to H\to p(H)<\Z^m,$$
where $p$ is the natural projection $G\times\Z^m\to \Z^m$ with $\ker(p)=G$. Then $H\cap \ker(p)$ is a subgroup of the slender group $G$, and hence it is finitely generated. So $H$ is also finitely generated.\\

(2) Since $M$ is a closed orientable hyperbolic 3-manifold,  $G=\pi_1(M)$  is a centerless, finitely generated Gromov hyperbolic group, and hence $G$ has $\fgfpa$. Moreover, $H_1(M)$ finite implies that the derived subgroup $[G, G]$ is finite index in the finitely generated group $G$ and hence it is again finitely generated.  Note that the weakly Howson property in the proof of Theorem \ref{main thm 2}(2), is only used to ensure that
$$\fix\psi\cap\alpha^{-1}((\id-\mathcal{L})\mathbb{Z}^m)$$
in Eq. (\ref{eq. 1}) is finitely generated. But now, since $\iso(M)$ is of odd order, $M$ has no involution.  By Proposition \ref{LW result}, the fixed subgroup $\fix\psi$ is inert in $G$ for every automorphism $\psi:G\to G$.  So the rank
$$\rk(\fix\psi\cap\alpha^{-1}((\id-\mathcal{L})\mathbb{Z}^m))\leq\rk(\alpha^{-1}((\id-\mathcal{L})\mathbb{Z}^m)))<\infty,$$
that is, $\fix\psi\cap\alpha^{-1}((\id-\mathcal{L})\mathbb{Z}^m)$ is finitely generated.

In conclusion, we can show that $G\times \mathbb{Z}^m$ has $\fgfpa$, by the same argument as in the proof of Theorem \ref{main thm 2}(2), without using the weakly Howson property.
\end{proof}
There are many closed hyperbolic 3-manifolds  such that their  fundamental groups satisfy condition (2) of Theorem \ref{main thm 3}. For example,

\begin{exam}
Let $M$ be the  3-manifold $\mathbb{S}^3-9_{32}$, and $M_{p,q}$ be the Dehn filling from $M$ along the slope $p \mathcal{M}+q\mathcal{L}$.  Here $9_{32}$ is the knot in Rolfsen's list \cite{Ro76}, and $(\mathcal{M},\mathcal{L})$ is  the canonical meridian-longitude system of the cusp of $M$. So $H_{1}(M)= \mathbb{Z}_{p}$.  From Snappy, we know that $\mathbb{S}^3-9_{32}$ is hyperbolic with trivial isometry group. When $p$, $q$ are large enough, $M_{p,q}$ is a closed hyperbolic 3-manifold by Thurston's hyperbolic Dehn surgery theorem.  The proof  of Thurston's hyperbolic Dehn surgery theorem  also implies the Dehn filling has a very short  geodesic core  when $p$, $q$ are large. So  any isometry of $M_{p,q}$  will preserve the core of the Dehn filling, and hence will preserve $M$, it will be isotopic to the identity. In total, when $p$, $q$ are large enough, we have $M_{p,q}$ a closed hyperbolic 3-manifold with trivial isometry group, and finite first homology group. So  the   fundamental group of it  satisfies condition  (2) of Theorem \ref{main thm 3}.
\end{exam}

Finally, inspired by Theorem \ref{main thm 3}, we wonder whether the assumption ``centerless and with the weakly Howson property" in  Theorem \ref{main thm 2} can be removed or not. Namely,

\begin{ques}
Does $G\times\Z$ have $\fgfpa$ if the group $G$ has $\fgfpa$ and the derived subgroup $[G, G]$ is finitely generated?
\end{ques}


\vspace{6pt}

\noindent\textbf{Acknowledgements.}  The authors would like to thank Luis Mendon\c{c}a for pointing out an error in the first version of this note, and for providing Remark \ref{fanli}. The authors also thank Enric Ventura for his helpful comments, which led us to more general results of $G\times \Z^m$ instead of  $G\times \Z$. Furthermore, the authors thank the anonymous referee very much for his/her valuable and
detailed comments helped to greatly improve our original manuscript, especially for inspiring us to introduce the concept of the weakly Howson property in Theorem \ref{main thm 2}.


\end{document}